\documentclass[12pt]{amsart}
\usepackage{amssymb,mathrsfs,longtable,textcomp}
\usepackage[matrix,arrow,curve]{xy}
\usepackage{setspace}
\usepackage{multirow}
\usepackage{epsfig,color}
\usepackage{url}

\newcounter{cequation}[section]

\newtheorem{theorem}[cequation]{Theorem}
\newtheorem*{theorem*}{Theorem}
\newtheorem{lemma}[cequation]{Lemma}
\newtheorem{corollary}[cequation]{Corollary}

\newtheorem{question}[cequation]{Question}
\newtheorem{proposition}[cequation]{Proposition}

\theoremstyle{definition}

\newtheorem{definition}[cequation]{Definition}
\newtheorem*{definition*}{Definition}

\theoremstyle{remark}
\newtheorem{remark}[cequation]{Remark}

\makeatletter\@addtoreset{equation}{section}
\makeatletter\@addtoreset{section}{part}

\makeatother

\def \O {\mathcal{O}}

\def \CC {\Bbbk}

\def \P {\mathbb{P}}
\def \PP {\mathbb{P}}

\def \ZZ {\mathbb{Z}}

\def \Sing {\mathrm{Sing}\,}

\def \ge {\geqslant}
\def \le {\leqslant}

\title{Fano weighted complete intersections of large codimension}

\author{Victor Przyjalkowski and Constantin Shramov}

\address{\emph{Victor Przyjalkowski}
\newline
\textnormal{Steklov Mathematical Institute of RAS, 8 Gubkina street, Moscow 119991, Russia.
}
\newline
\textnormal{\texttt{victorprz@mi.ras.ru, victorprz@gmail.com}}}

\address{\emph{Constantin Shramov}
\newline
\textnormal{Steklov Mathematical Institute of RAS,
8 Gubkina street, Moscow 119991, Russia.
}
\newline
\textnormal{\texttt{costya.shramov@gmail.com}}}

\thanks{Victor Przyjalkowski was %partially
supported by %Laboratory of Mirror Symmetry NRU HSE, RF Government grant, ag. \textnumero~14.641.31.0001.
grant MD--30.2020.1.
Constantin Shramov was supported %by the Russian Academic Excellence Project ``5-100'' and 
by the Foundation for the
Advancement of Theoretical Physics and Mathematics ``BASIS''.
Both authors are Young Russian Mathematics award winners and would like to thank
its sponsors and jury.
}

\begin{document}

\begin{abstract}
We classify smooth Fano weighted complete intersections of large codimension.
\end{abstract}

\maketitle

\section{Introduction}
\label{section:intro}

Given a smooth Fano variety $X$ over an algebraically closed field $\CC$ of characteristic zero, we denote by $i_X$ the Fano index of $X$, that is, the maximal positive integer $i$
such that the canonical class $K_X$ is divisible by $i$ in the Picard group of~$X$.
It is well known that~\mbox{$i_X\le \dim X+1$}, see~\cite[Corollary 3.1.15]{IP99}.
The goal of this note is to prove the following (we refer the reader to~\cite{Do82} and~\cite{IF00},
or to \S\ref{section:preliminaries} below, for the relevant definitions).

\begin{theorem}
\label{theorem:big_codim}
Let $X\subset\PP=\PP(a_0,\ldots,a_N)$ be a smooth well formed Fano weighted complete intersection of dimension $n\ge 2$ and codimension $k=N-n$ that is not an intersection with a linear cone. The following assertions hold.
\begin{itemize}
\item[(i)] One has $k\le n-i_X+1$.

\item[(ii)] If $k=n-i_X+1$, then $X$ is a complete intersection of $n-i_X+1$ quadrics in~\mbox{$\PP=\P^N$}.

\item[(iii)] Suppose that $k=n-i_X\ge 2$ (so that in particular $n\ge 3$).
Then $X$ is a complete intersection of $n-i_X-1$ quadrics and a cubic in $\PP=\P^N$.
\end{itemize}
\end{theorem}

Note that the assumption that $X$ in Theorem~\ref{theorem:big_codim} is well formed, as well
as the assumption that
$X$ is not an intersection with a linear cone,  can be omitted provided that $X$ is sufficiently general and $n\ge 3$,
see Proposition~\ref{proposition:unconization} below.

\smallskip
We are grateful to M.\,Ovcharenko who found a mistake in the first version of our paper.

\section{Preliminaries}
\label{section:preliminaries}

Let $a_0,\ldots,a_N$ be positive integers. Consider the graded algebra~\mbox{$\CC[x_0,\ldots,x_N]$},
where the grading is defined by assigning the weights $a_i$ to the variables~$x_i$.
Put
$$
\P=\P(a_0,\ldots,a_N)=\mathrm{Proj}\,\CC[x_0,\ldots,x_N].
$$

The weighted projective space $\P$ is said to be \emph{well formed} if the greatest common divisor of any $N$ of the weights~$a_i$ is~$1$. Every weighted projective space is isomorphic to a well formed one, see~\cite[1.3.1]{Do82}.
A subvariety $X\subset \P$ is said to be \emph{well formed}
if~$\P$ is well formed and
$$
\mathrm{codim}_X \left( X\cap\mathrm{Sing}\,\P \right)\ge 2,
$$
where the dimension of the empty set is defined to be $-1$.

We say that a subvariety $X\subset\P$ of codimension $k$ is a \emph{weighted complete
intersection of multidegree $(d_1,\ldots,d_k)$} if its weighted homogeneous ideal in $\CC[x_0,\ldots,x_N]$
is generated by a regular sequence of $k$ homogeneous elements of degrees $d_1,\ldots,d_k$.
Note that $\P$ can be thought of as a weighted complete
intersection of codimension $0$ in itself.

\begin{theorem}[{see \cite[Proposition 8]{Di86}}]
\label{theorem:Sing-WF}
Let $X\subset \PP$ be a well formed weighted complete intersection.
Then
$$
\Sing X=X\cap\Sing\PP.
$$
\end{theorem}

The following property is an analog of smoothness for subvarieties in a weighted
projective space.

\begin{definition}[{see~\cite[Definition 6.3]{IF00}}]
\label{definition: quasi-smoothness}
Let $p\colon \mathbb A^{N+1}\setminus \{0\}\to \P$ be the natural projection to the weighted projective space. A subvariety $X\subset \P$
is called \emph{quasi-smooth} if the preimage~\mbox{$p^{-1}(X)$} is smooth.
\end{definition}

\begin{lemma}
\label{lemma:wellformization}
Let $X\subset \PP$ be a weighted complete intersection.
Then $X$ is isomorphic to a weighted complete intersection $X'$ of the same codimension in a well formed weighted projective space $\PP'$. Moreover, if $X$ is quasi-smooth, $X'$ is also quasi-smooth, and if $X$ is a general weighted
complete intersection of the corresponding multidegree in $\PP$, then $X'$
is also a general weighted
complete intersection of the corresponding multidegree in $\PP'$.
\end{lemma}
\begin{proof}
We may assume that the greatest common divisor of the weights $a_0,\ldots,a_N$ equals $1$.
Suppose that $\PP$ is not well formed. This means that, up to renumbering, one has~\mbox{$a_i=aa_i'$},
where $a_i'\in \ZZ_{\ge 1}$ for  $i=1,\ldots,N$, and $a\in\ZZ_{\ge 2}$.
Then, if we put $a_0'=a_0$, we get an isomorphism $\PP\cong \PP'=\PP(a_0',\ldots,a_N')$ given by the $a$-th Veronese map, see, for instance,~\cite[Lemma 5.7]{IF00}.
The image of $X$ under this isomorphism is a weighted complete intersection $X'\subset\PP'$. Indeed, by homogenicity
$a$ divides all degrees of the variable $x_0$ in the equations defining $X$, so equations that define $X'$ are given
from ones that define $X$ by replacing $x_0^a$ by $x_0'$ and $x_i$ by $x_i'$, $i=1,\ldots, N$,
where $x_i$ and $x_i'$ are weighted homogeneous coordinates on $\PP$ and $\PP'$ of weights $a_i$ and $a_i'$, respectively.
Repeating this procedure sufficiently many times, we may assume that $\PP'$ is well formed.

The affine cone over $X'$ is a quotient of the affine cone over $X$ by a group generated by quasi-reflections, so $X'$ is quasi-smooth if and only if $X$ is quasi-smooth. The assertion about the generality of $X'$ is obvious.
\end{proof}

Although quasi-smoothness is a natural condition to consider, we will be mostly interested
in weighted complete intersections that are smooth in the usual sense.

\begin{remark}[{\cite[Corollary 2.14]{PrzyalkowskiShramov-Weighted}}]
\label{remark:smooth}
Let $X$ be a smooth well formed weighted complete intersection.
Then $X$ is quasi-smooth.
\end{remark}

The following lemma provides a smoothness criterion for weighted complete intersections.

\begin{lemma}[{\cite[Lemma 2.5]{PrzyalkowskiShramov-Weighted}}]
\label{lemma:smoothness-criterion}
Let $X\subset\PP$ be a smooth well formed weighted complete intersection
of multidegree $(d_1,\ldots,d_k)$.
Then for every $r$ and every choice of~$r$ weights~\mbox{$a_{i_1},\ldots,a_{i_r}$, $i_1<\ldots<i_r$},
such that their greatest common divisor $\delta$ is greater than~$1$,
there exist $r$ degrees $d_{s_1},\ldots,d_{s_r}$, $s_1<\ldots<s_r$,
such that their greatest common divisor is divisible by~$\delta$.
\end{lemma}

The following definition describes weighted complete intersections
that are to a certain extent analogous to complete intersections
in a usual projective space that are contained in a hyperplane.

\begin{definition}[{cf. \cite[Definition~6.5]{IF00}}]
\label{definition:cone}
A weighted complete intersection~$X\subset\P$
is said to be \emph{an intersection
with a linear cone} if one has $d_j=a_i$ for some~$i$ and~$j$.
\end{definition}

\begin{theorem}[{see \cite[Theorem~6.17]{IF00}}]
\label{theorem:cone-vs-wf}
Let $X\subset\PP$ be a quasi-smooth weighted complete intersection of dimension at least $3$.
Suppose that $\PP$ is well formed and $X$ is not an intersection with a linear cone, and that
$X$ is general in the family of weighted complete intersections of the same multidegree in $\PP$.
Then $X$ is well formed.
\end{theorem}

\begin{remark}
Theorem~\ref{theorem:cone-vs-wf} does not hold in dimensions less than $3$. A counterexample is
a hypersurface of degree $9$ in $\PP(1,2,2,3)$, which is smooth and quasi-smooth, but not well formed,
see~\cite[6.15(ii)]{IF00}.
\end{remark}

The following result shows that a general enough quasi-smooth weighted complete intersection
is isomorphic to a weighted complete intersection with nice properties.

\begin{proposition}
\label{proposition:unconization}
Let $X\subset \PP$ be a quasi-smooth weighted complete intersection of dimension at least $3$.
Assume that $X$ is general in the family of weighted complete intersections of the same multidegree in $\PP$.
Then there exists a quasi-smooth well formed weighted complete intersection $X'$ isomorphic to $X$ which is not an intersection with a linear cone.
\end{proposition}

\begin{proof}
Suppose that $X$ is an intersection with a linear cone.
Let $(d_1,\ldots, d_k)$ be the multidegree of $X$. We may assume that~\mbox{$d_k=a_N$}.
Let~\mbox{$x_i$, $i=0,\ldots, N$}, be a weighted homogeneous coordinate of weight $a_i$ on $\PP$.
Since~$X$ is general in the family of weighted complete intersections of the same multidegree in $\PP$, there exists a homogeneous polynomial of degree~$d_k$ in the weighted homogeneous ideal
of $X$ of the form~\mbox{$f_k=f-x_N$}, where the polynomial~$f$ depends only on variables~\mbox{$x_0,\ldots, x_{N-1}$}.
Moreover, substituting $f$ instead of $x_N$ into the remaining polynomials $f_1,\ldots,f_{k-1}$ that generate the ideal of $X$,
one can assume that these polynomials do not depend on $x_{N}$.
Thus the natural projection
\begin{equation}
\begin{gathered}\label{eq:projection}
\PP(a_0,\ldots,a_{N-1}, a_N)\dasharrow \PP'=\PP(a_0,\ldots, a_{N-1}), \\ (x_0:\ldots:x_{N-1}:x_N)\mapsto
(x_0:\ldots:x_{N-1}),
\end{gathered}
\end{equation}
induces
the isomorphism of $X$ with a weighted complete intersection $X'$ of multidegree~\mbox{$(d_1,\ldots,d_{k-1})$}
in $\PP'$.
Quasi-smoothness of $X$ means that the Jacobian matrix of the polynomials defining $X$ is non-degenerate.
Since the variable $x_N$ appears linearly and only in $f_k$, quasi-smoothness of $X$
implies that $X'$ is quasi-smooth as well. By the assumption on the generality of $X$, we conclude that
$X'$ is a general weighted complete intersection of multidegree~\mbox{$(d_1,\ldots,d_{k-1})$}
in $\PP'$.

Repeating the above procedure sufficiently many times, we may assume that $X'$ is not an intersection with a linear cone.
Applying Lemma~\ref{lemma:wellformization}, we may assume that $X'$ is a quasi-smooth weighted complete intersection
in a well formed weighted projective space; however, as a result of the application of Lemma~\ref{lemma:wellformization},
it may become an intersection with a linear cone again. Alternating these two steps, we will
finally arrive to a quasi-smooth weighted complete intersection $X'$ in a well formed weighted projective space that is
not an intersection with a linear cone, because projection~\eqref{eq:projection} decreases the dimension of the ambient
weighted projective space, while the application of Lemma~\ref{lemma:wellformization} leaves it the same.
It remains to apply Theorem~\ref{theorem:cone-vs-wf} to conclude that $X'$ is well formed.
\end{proof}

If a weighted complete intersection $X$ is smooth and well formed, then
$X$ is quasi-smooth by Remark~\ref{remark:smooth}.
However a priori we cannot assume smoothness instead of quasi-smoothness in Proposition~\ref{proposition:unconization}
since $X$ is not necessarily well formed.

\begin{question}
Does Proposition~\ref{proposition:unconization} nevertheless hold if we replace quasi-smoothness by smoothness?
\end{question}

\section{Proof of the main result}
\label{section:proof}

In this section we prove Theorem~\ref{theorem:big_codim}.
By $\PP$ we will always denote the weighted projective space
$\PP(a_0,\ldots, a_N)$. We will say that a weighted complete intersection~\mbox{$X\subset\P$}
of multidegree~\mbox{$(d_1,\ldots,d_k)$} is \emph{normalized} if the inequalities~\mbox{$a_0\le\ldots\le a_N$} and $d_1\le\ldots\le d_k$ hold. Recall that if $X\subset\PP$ is a smooth well formed Fano weighted complete intersection of
multidegree $(d_1,\ldots,d_k)$ and dimension at least $2$, then $i_X=\sum a_i-\sum d_j$,
see~\mbox{\cite[Corollary~2.8]{PSh19}}. Note that
in a more general situation the following version of adjunction
formula holds.

\begin{theorem}[{see~\cite[Theorem 3.3.4]{Do82}, \cite[6.14]{IF00}}]
\label{theorem:adjunction}
Let $X$ be a quasi-smooth
well formed weighted complete intersection.
Let $\omega_X$ be the dualizing sheaf on $X$.
Then
$$
\omega_X=\O_X\left(-i_X\right).
$$
\end{theorem}
In particular, Theorem~\ref{theorem:adjunction} shows that
if $X$ is a quasi-smooth well formed weighted complete intersection of arbitrary dimension, then
$X$ is Fano if and only if $i_X>0$.

In the proof of Theorem~\ref{theorem:big_codim} we will use the following auxiliary results.

\begin{theorem}[{see \cite[Theorem~1.2]{PST}}]
\label{theorem:PST}
Suppose that $X\subset\P$ is a normalized smooth well formed Fano or Calabi--Yau weighted complete intersection of codimension $k$ that is not an intersection with a linear cone. Then  $a_{k-1}=1$. Moreover, a general element of the linear system $|\O_X(1)|$ is smooth.
\end{theorem}

The following result is implicitly contained in \cite{PST}. We prove it for the reader's convenience.

\begin{lemma}
\label{lemma:PST}
Suppose that $X\subset\P$ is a normalized smooth well formed Fano weighted complete intersection
of dimension $n$ multidegree $(d_1,\ldots,d_k)$ that is not an intersection with a linear cone.
Let $X'$ be a general element of the linear system $|\O_X(1)|$. Then~$X'$
is a well formed weighted complete intersection of multidegree $(d_1,\ldots,d_k)$ in~\mbox{$\PP'=\PP(a_1,\ldots, a_N)$}.
\end{lemma}
\begin{proof}
By the classification of one-dimensional smooth well formed Fano weighted complete intersections
(see for instance \cite[Lemma~2.6]{PSh19}), we may assume that $\dim X\ge 2$.

The linear system $|\O_X(1)|$ is non-empty by Theorem~\ref{theorem:PST}, and
every element of~\mbox{$|\O_X(1)|$} is cut out on $X$ by a weighted hypersurface of degree $1$ in $\PP$
(see, for instance,~\mbox{\cite[Corollary~3.3]{PSh19}}). Therefore, one can consider $X'$ as a weighted complete intersection of multidegree $(d_1,\ldots,d_k)$ in $\PP'$.

Suppose that the weighted projective space $\PP'$ is not well formed.
Then there are $N-1$ numbers among $a_1,\ldots,a_N$ such that their greatest common divisor is
at least $2$. We may assume that these are $a_2,\ldots,a_N$.
Let $x_i$ be the weighted homogeneous coordinate of weight $a_i$ on $\PP$, and let $\Lambda\subset\PP$
be the subset given by $x_0=x_1=0$. Then $\dim\Lambda=N-2$. Since $\dim X\ge 2$, the intersection $X\cap\Lambda$
is not empty. On the other hand, by~\cite[5.15]{IF00}
the weighted projective space $\PP$ is singular along $\Lambda$, and by Theorem~\ref{theorem:Sing-WF}
the weighted complete intersection $X$ is singular along
$X\cap\Lambda$. The obtained contradiction shows that~$\PP'$ is well formed.

By \cite[5.15]{IF00} one has $\Sing\PP'\subset\Sing\PP$. On the other hand,
since $X$ is smooth and well formed, it is disjoint from $\Sing\PP$ by Theorem~\ref{theorem:Sing-WF}.
Hence $X'$ is disjoint from  $\Sing\PP$, and thus also from $\Sing\PP'$,
which implies that $X'$ is well formed.
\end{proof}

\begin{corollary}
\label{corollary:PST}
Suppose that $X\subset\P$ is a normalized smooth well formed Fano or Calabi--Yau weighted complete intersection of codimension $k$ that is not an intersection with a linear cone. Then
$a_{k+i_X-1}=1$.
\end{corollary}

\begin{proof}
If $X$ is Calabi--Yau, the assertion follows from Theorem~\ref{theorem:PST}.
Thus, we assume that $X$ is Fano. Let $n=\dim X$.
By  \cite[Lemma~2.6]{PSh19}, we may assume that $\dim X\ge 2$.
Furthermore, by \cite[Corollary~3.8(i),(ii)]{PSh19}, we may assume that $i_X\le n-1$.

Applying Theorem~\ref{theorem:PST} consecutively $i_X$ times and keeping in mind that
a general element of $|\O_X(1)|$ is well formed by Lemma~\ref{lemma:PST},
we see that $a_0=\ldots=a_{i_X-1}=1$, and
there exists a normalized smooth well formed Calabi--Yau
weighted complete intersection of codimension $k$ in
$\PP(a_{i_X},\ldots,a_N)$, cf. Theorem~\ref{theorem:adjunction}. Also, the obtained weighted
complete intersection is not an intersection with a linear cone. It remains to apply
Theorem~\ref{theorem:PST} once again to conclude that $a_{i_X}=\ldots=a_{k+i_X-1}=1$.
\end{proof}

\begin{theorem}[{\cite[Lemma 18.4]{IF00}}]
\label{theorem:deltas}
Suppose that $X\subset\P$ is a normalized quasi-smooth well formed weighted complete intersection of multidegree $(d_1,\ldots,d_k)$ that is not an intersection with a linear cone.
Then $d_{k-j}> a_{N-j}$ for all $1\le j\le k$.
\end{theorem}

\begin{lemma}[{see~\cite[Proposition 3.1(1)]{ChenChenChen}}]
\label{lemma:last_weight}
Suppose that $X\subset\P$ is a normalized quasi-smooth  well formed weighted complete intersection
of multidegree $(d_1,\ldots,d_k)$  that is not an intersection with a linear cone. Then $d_{k}\ge 2a_{N}$.
\end{lemma}

The following is a particular case of Theorem~\ref{theorem:big_codim}.

\begin{lemma}[{cf. \cite[Proof of Theorem 1.3]{ChenChenChen}}]
\label{lemma:Fano_dim-codim}
Suppose that $X\subset\P$ is a smooth  well formed Fano weighted complete intersection of dimension $n$ and
codimension $k$  that is not an intersection with a linear cone.
Then~\mbox{$k\le n$}. Moreover, if $k=n$, then $X$ is a complete intersection of $n$ quadrics in~\mbox{$\PP=\PP^{2n}$}.
\end{lemma}

\begin{proof}
We may assume that $X$ is normalized. Suppose that $k\ge n$. Then~\mbox{$a_0=\ldots=a_n=1$} by Corollary~\ref{corollary:PST}.
Let $(d_1,\ldots,d_k)$ be the multidegree of $X$. Using Theorem~\ref{theorem:deltas}, we obtain
$$
1\le i_X=\sum_{i=0}^N a_i-\sum_{j=1}^k d_j=a_0+\ldots+a_n+(a_{n+1}-d_1)+\ldots+(a_N-d_k)\le n+1-k\le 1.
$$
This implies that $k=n$, and
$d_j-a_{n+j}=1$ for all $1\le j\le k$. In particular, by Lemma~\ref{lemma:last_weight}
we have
$$
1=d_k-a_N\ge a_N,
$$
so that $a_N=1$ and $d_k=2$. This means that $a_i=1$ for all $0\le i\le N$
and $d_j=2$ for all~\mbox{$1\le j\le k$}.
\end{proof}

\begin{remark}
In fact, the first assertion of Lemma~\ref{lemma:Fano_dim-codim} holds under a weaker assumption
that $X$ is a quasi-smooth well formed Fano weighted complete intersection, see~\mbox{\cite[Theorem 1.3]{ChenChenChen}}.
However, quasi-smoothness is not enough for the second assertion to hold.
For instance, a general hypersurface of degree $30$ in $\PP(6,10,15)$ is a quasi-smooth
well formed Fano curve for which such an assertion fails.
\end{remark}

Similarly to Lemma~\ref{lemma:Fano_dim-codim}, we prove the following result
(which is also a particular case of Theorem~\ref{theorem:big_codim}).

\begin{lemma}
\label{lemma:Fano_dim-codim-next}
Suppose that $X\subset\P$ is a smooth  well formed Fano weighted complete intersection of dimension $n$ and
codimension $k$ that is not an intersection with a linear cone.
Suppose that $k=n-1\ge 2$. Then~$X$ is either a complete intersection of $n-1$ quadrics in $\PP=\PP^{2n-1}$,
or a complete intersection of $n-2$ quadrics and a cubic in~\mbox{$\PP=\PP^{2n-1}$}.
\end{lemma}

\begin{proof}
We may assume that $X$ is normalized.  Then~\mbox{$a_0=\ldots=a_{n-1}=1$} by  Corollary~\ref{corollary:PST}.
Let $(d_1,\ldots,d_k)$ be the multidegree of $X$. Using Theorem~\ref{theorem:deltas}, we obtain
\begin{multline}
\begin{gathered}
\label{equarion:k=n-1}
1\le i_X=\sum_{i=0}^N a_i-\sum_{j=1}^k d_j=a_0+\ldots+a_{n-1}+a_n+(a_{n+1}-d_1)+\ldots+(a_N-d_k)=\\
=a_0+a_1+(a_{n+1}+a_2-d_1)+\ldots+(a_{N-1}+a_{n-1}-d_{k-1})+(a_N+a_n-d_k)=\\
=2+(a_{n+1}+1-d_1)+\ldots+(a_{N-1}+1-d_{k-1})+(a_N+a_n-d_k)\le\\ \le 2+(a_N+a_n-d_k).
\end{gathered}
\end{multline}

Suppose that $a_N=1$.
This means that $a_i=1$ for all $0\le i\le N$, so that $X$ is a complete intersection in the usual projective space.
From~\eqref{equarion:k=n-1}
we conclude that $d_k<4$.
If $d_k=2$, then $X$ is a complete intersection of $k=n-1$ quadrics. If $d_k=3$, then
we look at~\eqref{equarion:k=n-1} once again and see that
for all $1\le j\le k-1$ one has $d_{j}=a_{n+j}+1=2$. This means that~$X$ is a complete intersection of $k=n-2$ quadrics
and a cubic.

Now suppose that $a_N>1$.
If $a_n=a_N$, then by Lemma~\ref{lemma:smoothness-criterion}
at least $k+1$ degrees among~\mbox{$d_1,\ldots,d_k$} are divisible by $a_n$, which is absurd.
Thus $a_n<a_N$.
From~\eqref{equarion:k=n-1} we have
$$
1\le 2+(a_{n}+a_N-d_k)<2+(2a_N-d_k)\le 2
$$
by Lemma~\ref{lemma:last_weight}.
This means that $d_k=2a_N$ and $a_N=a_n+1$.
Looking at~\eqref{equarion:k=n-1} once again and keeping in mind that $k\ge 2$ by assumption, we see that
\begin{equation}\label{eq:d1anpl1}
d_1=a_{n+1}+1.
\end{equation}

Since $a_N=a_n+1$, one has either $a_{n+1}=a_n$ or $a_{n+1}=a_N$.
In the former case~\eqref{eq:d1anpl1} gives $d_1=a_{n+1}+1=a_{n}+1=a_N$. This contradicts the assumption that $X$ is not an intersection with a linear cone.
In the latter case by Lemma~\ref{lemma:smoothness-criterion} the number
$a_N$ divides all $k$ degrees $d_j$. Since $X$ is not an intersection
with a linear cone, we see that $d_j\ge 2a_N$ for all $1\le j\le k$.
Using~\eqref{eq:d1anpl1}, we get
$$
2a_N\le d_1=a_{n+1}+1=a_N+1,
$$
which is again a contradiction.
\end{proof}

Now we are ready to prove Theorem~\ref{theorem:big_codim} in full generality.

\begin{proof}[Proof of Theorem~\ref{theorem:big_codim}]
Assume that $i_X\ge 2$.
By Theorem~\ref{theorem:PST}, the linear system $|\O_X(1)|$ is not empty.
Let $X'$ be a general divisor from $|\O_X(1)|$.
Then $X'$ is smooth by Theorem~\ref{theorem:PST}.
Moreover, $X'$ is a well formed
weighted complete intersection of multidegree $(d_1,\ldots,d_k)$ in $\PP'=\PP(a_1,\ldots, a_N)$ by Lemma~\ref{lemma:PST},
and $X'$ is not an intersection with a linear cone. Furthermore, $X'$ is Fano by Theorem~\ref{theorem:adjunction}.
Observe that Theorem~\ref{theorem:big_codim} holds for $X$ if and only if it holds for $X'$.
Note also that if the dimension of $X'$ is at most $3$, then the assertion
of Theorem~\ref{theorem:big_codim} holds for $X'$ by
the classification of smooth Fano weighted complete intersections in low dimensions (see,
for instance,~\mbox{\cite[Tables 1 and 2]{PSh18}}).

Therefore, repeating the above procedure several times and keeping in mind Theorem~\ref{theorem:adjunction}, we may assume that $i_X=1$,
since both $k$ and $n-i_X+1$ do not change when we take the section.
In this case we know that $k\le n=n-i_X+1$ by Lemma~\ref{lemma:Fano_dim-codim}, which proves
assertion~(i). Furthermore, if $k=n=n-i_X+1$, then $X$ is a complete intersection of $n$ quadrics
in $\PP^{2n}$ by Lemma~\ref{lemma:Fano_dim-codim-next}, which proves
assertion~(ii). Finally, if~\mbox{$k=n-1=n-i_X$} and $k\ge 2$, then
$X$ is a complete intersection of $n-2$ quadrics
and a cubic in $\PP^{2n-1}$ by Lemma~\ref{lemma:Fano_dim-codim-next}, which proves
assertion~(iii).
\end{proof}

\begin{remark}
Assertion~(iii) of  Theorem~\ref{theorem:big_codim} does not hold without the assumption
that~\mbox{$k\ge 2$}. Indeed, except for a cubic hypersurface in $\P=\P^{n+1}$, there exist two other
types of smooth well formed Fano weighted hypersurfaces
of dimension $n$ and Fano index~\mbox{$n-1$}: a hypersurface
of degree $6$ in $\P=\P(1^n, 2,3)$,
and a hypersurface of degree $4$ in~\mbox{$\P=\P(1^{n+1}, 2)$}, cf. \cite[Corollary~3.8(iii)]{PSh19}.
\end{remark}

It seems that the analog of Theorem~\ref{theorem:big_codim} for $k=n-i_X-1$ does not look that nice.
For instance, there are $5$ families of five-dimensional weighted complete intersections of such codimension,
as well as $5$ families of six-dimensional weighted complete intersections.
One of the latter varieties is a six-dimensional weighted complete intersection of multidegree~\mbox{$(2,2,2,6)$} in~\mbox{$\PP(1^{10},3)$}.

\begin{question}
Is there a finite number of series of smooth well formed weighted complete intersections of codimension $n-i_X-1$?
\end{question}

\end{document}